\documentclass{article}
\usepackage[utf8]{inputenc}
\usepackage{amsfonts}
\usepackage{amsmath,amsthm}
\usepackage{amscd}
\usepackage[all,cmtip]{xy}
\usepackage{amssymb}
\usepackage{indentfirst}
\usepackage{geometry}
\usepackage{fancyhdr}
\usepackage{lipsum}
\usepackage{mathtools}
\usepackage{relsize}
\usepackage{physics}
\usepackage{braket}
\usepackage{bm}
\usepackage{esint}
\usepackage{enumitem}
\usepackage{appendix}
\usepackage{mathrsfs}
\usepackage{xcolor}
\usepackage{bbold}
\usepackage{comment}
\usepackage[
backend=biber,
style=alphabetic,
maxbibnames=99
]{biblatex} %Imports biblatex package
\allowdisplaybreaks

\newtheorem{theorem}{Theorem}

\newtheorem{assumption}[theorem]{Assumption}

\newtheorem{conjecture}[theorem]{Conjecture}
\newtheorem{corollary}[theorem]{Corollary}

\newtheorem{definition}[theorem]{Definition}

\newtheorem{proposition}[theorem]{Proposition}
\newtheorem{remark}[theorem]{Remark}

\numberwithin{theorem}{section}

\def\text{\mbox}

\def\bi{\begin{itemize}}
\def\ei{\end{itemize}}
\def\bem{\begin{emuerate}}
\def\eem{\end{enumerate}}
\def\beq{\begin{eqnarray*}}
\def\eeq{\end{eqnarray*}}

\def\NN{\mathbb{N}} %naturals
\def\RR{\mathbb{R}} %reals
\def\PP{\mathbb{P}} %probability
\def\EE{\mathbb{E}} %expectation
\def\TT{\mathbb{T}} %Torus
\def\HH{\mathbb{H}} %Hyperbolic Space or Quaternions

\def\mc{\mathcal}

\def\1{\mathbb{1}}

\newcommand*\Laplace{\mathop{}\!\mathbin\bigtriangleup}%Laplacian Operator
\title{Martin boundary and invariant fields for multiplicative SHE}
\author{Hongyi Chen\footnote{Department of Mathematics, Aarhus University, Denmark. Email: hchen77@math.au.dk}}

\begin{document}
\maketitle
\begin{abstract}
    \noindent We study invariant measures of nonlinear multiplicative stochastic heat equations in the weak disorder regime. Under a natural second-moment condition, we show that positive invariant measures are in one-to-one correspondence with bounded positive harmonic functions of the underlying space. This implies that the space of invariant measures inherits the structure of the Martin boundary. We also show that whenever the deterministic semigroup flow converges to a bounded harmonic function, the stochastic evolution converges to the corresponding invariant measure. The results apply to many settings with nontrivial Martin boundary, such as negatively curved manifolds and trees.
\end{abstract}
\tableofcontents
\section{Introduction}\label{sec:intro}
Let $(M,d,\mu)$ be a metric measure space equipped with a symmetric Dirichlet form $\mc{E}$. Denote the heat semigroup and $L^2(M)-$self-adjoint operator associated with $\mc{E}$ by $P_t$ and $\mc{L}$, respectively.
We study the multiplicative stochastic heat equation, formally written as \begin{equation}\label{eq:SHE}
\partial_t u(t,x)=\mc{L} u(t,x)+\beta\,f(u(t,x))\,\dot W(t,x), \quad t>0,x\in M.
\end{equation}
Here, \(\beta>0\) is the inverse temperature parameter, $f:\RR\to\RR$ is a Lipschitz function with Lipschitz paramter $L_f,$ that is, \(|f(u)-f(\tilde{u})|\leq L_f|u-\tilde{u}|\) and $f(0)=0$; \(\dot{W}=\dot{W}(t,x)\) is a centered Gaussian noise with covariance kernel $R$ (the exact definition will be given in section \ref{sec:prelim}). We will assume $R$ satisfies the following assumption for this paper.
\begin{assumption}\label{assump:WD}
    Let $p_t$ denote the heat kernel, i.e. the kernel of the heat semigroup $P_t$, which we will assume exists. We assume there exists $\Lambda\in(0,\infty)$ such that \[\sup_{x,x'\in M}\int_0^{\infty}\iint_{M^2}p_t(x,y)p_t(x',y')|R(y,y')|dydy'dt=\Lambda.\]
\end{assumption}

\noindent We say any $h:M\to\RR$ satisfying $\mc{L}h=0$ is a harmonic function.
\begin{theorem}\label{thm:main}
Let $u_0\in L^\infty(M)$ and suppose Assumption~\ref{assump:WD} holds. Let $u$ denote the solution of \eqref{eq:SHE} starting from $u_0$.

\begin{enumerate}
\item[(i)] If $(\beta L_f)^2\Lambda<1$, then $u$ is in the weak disorder regime,
\[
\sup_{t\ge0,x\in M}\EE[u(t,x)^2]<\infty.
\]

\item[(ii)] Under the same condition, the set of stationary solutions of \eqref{eq:SHE} $Z$ (see Definition \ref{def: stationary sol and inv meas}) satisfying for one (and thus all) $t\in\RR$
\[
\sup_{x\in M}\EE[Z(t,x)^2]<\infty
\]
are in one-to-one correspondence with bounded harmonic functions $\mathrm{Har}_b(M)$.\\
More precisely, for each $h\in \mathrm{Har}_b(M)$ there exists a random field \(\set{Z^h(x)}_{x\in M}\) whose law is the unique invariant measure of \eqref{eq:SHE} at inverse temperature $\beta$ satisfying $\EE[Z^h(x)]=h(x)$ and \[u^h(t,\cdot)\stackrel{t\uparrow\infty}{\longrightarrow}Z^h\quad \text{ in }L^2(\Omega),\] where $u^h(t,x)$ is the solution of \eqref{eq:SHE} starting from $h$.
\end{enumerate}
\end{theorem}

\begin{remark}[On the moment assumption]
The uniform pointwise second-moment bound imposed in Theorem~\ref{thm:main} is the natural condition corresponding to the weak disorder/diffusive regime. In the linear case $f(u)=u$, it indicates that the partition function of the associated polymer has uniformly bounded second moments, reflecting the absence of localization and the persistence of diffusive behavior. The class of invariant random fields considered here therefore corresponds to physically relevant stationary polymer states in the weak disorder phase.
\end{remark}

\begin{remark}
    The well-posedness and moment bound condition does not depend on the requirement $f(0)=0$. However, $f(0)=0$ does imply that $u\equiv 0$ is a solution, is known to be connected to positivity preservation\cite{CH19}, and it is needed to makes the arguments we use work (see also \cite[Section 5]{COTX25}). After the first version of this work was released, \cite{JKKM26} was able to classify all annealed, ergodic invariant measures without this extra assumption on $f$. We will discuss this in the literature review.
\end{remark}

Next is a fluctuation result which differs across harmonic sectors.
\begin{theorem}\label{thm:fluc}
Let $(\beta L_f)^2\Lambda<1$ and $h\in\mathrm{Har}_b(M)$. Let $Z_\beta^h$ be the stationary solution of \eqref{eq:SHE} at inverse temperature $\beta$ with mean $h$ at time 0 (thus any time $t\in \RR$) driven by some two-sided noise $W$. Let $G_h$ be a random field on $M$ defined by
\[
G_h(x)=\int_0^\infty\int_M p_{s}(x,y)f(h(y))\,W(dy,ds).
\]
Then
\[
\frac{Z_\beta^h(0,\cdot)-h}{\beta}\;\stackrel{\beta\downarrow0}{\longrightarrow}G_h
\quad\text{in }L^2(\Omega).
\]
\end{theorem}
\begin{remark}
    By the definition of stationary solutions (Definition \ref{def: stationary sol and inv meas}), the choice of $0$ in the above theorem statement is arbitrary and can be replaced by any $t\in \RR$. Also, it obviously implies an analogous statement about the invariant measure.
\end{remark}
Finally, we show that the invariant measures identified in Theorem~\ref{thm:main} completely determine the long time dynamics of the solution $u$ under mild assumptions on the initial condition.
\begin{theorem}\label{thm:IC determine long time}
Suppose $u_0\in L^\infty(M)$ satisfies
\[
P_tu_0 \to h\in \mathrm{Har}_b(M)
\quad\text{in }C(M)
\]
as $t\to\infty$, and $(\beta L_f)^2\Lambda<1$. Then as $t\uparrow\infty$, the solution of \eqref{eq:SHE} starting from $u_0$, denoted by $u$, converges in $L^2(\Omega)$ to a random field $\set{Z_\beta^h(x)}_{x\in M}$, whose law is an invariant measure with mean $h$ (again see Definition \ref{def: stationary sol and inv meas}).
\end{theorem}

The paper is organized as follows. We will we list some useful preliminary definitions and fix some notation, discuss existing literature and discuss the results on spaces with nontrivial Martin boundary to finish the introduction. We will then move on to prove the main theorem in section \ref{sec:pf main}. Item (i) requires almost no new mathematics, so most of the section is reserved for proving item (ii), the classification result. For that, we first show existence given a bounded positive harmonic $h$ result using an argument originating in \cite{GL20} (see also \cite{GHL25,COTX25}), then we conclude by showing uniqueness per $h$ by a chaos expansion inspired contraction. We then move onto proving Theorems \ref{thm:fluc} and \ref{thm:IC determine long time} in sections \ref{sec:fluc} and \ref{sec:long time}, respectively.

\subsection{Preliminaries}\label{sec:prelim}
We will first fix some notation and conventions used for the rest of the paper.
\begin{itemize}
    \item We use $C,c$. to denote generic constants that are independent of quantities of interest. The exact values of these constants may change from line to line.
    \item $\mathrm{Har}_b(M)$ will denote the space of bounded harmonic functions on $M$, and $\mathrm{Har}^+_b(M)$ denotes the subset of positive bounded harmonic functions.
    \item $dz$ will denote integration with respect to the reference measure $\mu$ on $M$ in the variable $z$.
\end{itemize}
We now introduce centered Gaussian noises with covariance kernel $R$ and rigorously define the notion of a solution. Any positive-definite function $R$ on $M$ defines an inner product by \[\braket{\phi,\psi}_R:=\int_{0}^{\infty}\iint_{M^2}\phi(t,x)R(x,y)\psi(t,y)dxdydt.\]
We denote the Hilbert space of space-time functions corresponding to $\braket{\cdot,\cdot}_R$ by $\mathcal{H}_R$. On a probability space $(\Omega,\mathcal{G},\PP)$, the noise $\dot W$ is defined as the isonormal Gaussian process\cite{Nualart95TheMC} over $\mathcal{H}_R$ such that $\mathcal{G}=\sigma(\dot W)$. We will not denote its $R-$dependence as there is no concern for confusion.\\
The stochastic integral of a random field $F$ with respect to $W$, denoted by \[I(F):=\int_0^\infty\int_M F(t,x) W(dx,dt),\]
is defined for any $F$ such that \[\EE[I(F)^2]=\int_0^\infty\iint_{M^2}\EE[F(t,x)F(t,y)]R(x,y)dxdydt<\infty.\]
We call such $F$ the $W-$integrable random fields. For any $F,\widetilde{F}$ which are $W-$integrable, we have \[\EE[I(F)I(\widetilde F)]=\int_0^\infty\iint_{M^2}\EE[F(t,x)\widetilde F(t,y)]R(x,y)dxdydt.\]
We call the above formulas It\^o--Walsh isometry.

\begin{definition}
    \label{def of Ito sol}
A process $u=\{u(t,x); (t,x)\in\mathbb{R}_{+}\times M \}$ is called a random field solution of~\eqref{eq:SHE} in the It\^o-Walsh sense \cite{Walsh86} if the following conditions are met:
\begin{enumerate}
\item $u$ is adapted;
\item $u$ is jointly measurable with respect to $\mathcal{B}(\mathbb{R}_{+}\times M )\otimes \mathcal{G}$; 
\item $\mathbf{E}(u(t,x)^2)<\infty$ for all $(t,x)\in\mathbb{R}_{+}\times M$;
\item The function $(t,x)\to u(t,x)$ is continuous in $L^2(\Omega)$;
\item $u$ satisfies \[\label{eq:SHE-mild}u(t,x)=P_tu_0(x)+\beta\int_0^t\int_M p_{t-s}(x,y)f(u(s,y))W(ds,dy)\] almost surely for all $(t,x)\in\mathbb{R}_{+}\times M$.
\end{enumerate}
\end{definition}

\begin{remark}
    We note that by It\^o--Walsh isometry, the second moment of the solution $u$ starting from $u_0$ has the useful representation \[\EE[u(t,x)]=P_tu_0(x)^2+\beta^2\int_0^t \iint_{M^2}p_{t-s}(x,y)p_{t-s}(x,y')R(y,y')\EE[f(u(s,y))f(u(s,y'))]dydy'ds.\]
    Assumption \ref{assump:WD} implies that for all $T>0$, we have \[\sup_{x,x'\in M}\int_0^T\iint_{M^2}p_t(x,y)p_t(x',y')|R(y,y')|dydy'dt<\infty.\]
    By standard arguments (c.f. \cite{Dalang98,ChenKim19,BOTW-Heisenberg,BCHOTW}) using the above moment representation, \eqref{eq:SHE} with initial $u_0\in L^\infty(M)$ admits a unique solution.
\end{remark}

We now define the notions of invariant measures and stationary solutions, which clarifies the statements of Theorem \ref{thm:main}. 

\begin{definition}\label{def: stationary sol and inv meas}
    We extend the noise $W$ to the indefinite past by simply extending the temporal domain of the stochastic integral $I$ from $\RR_+$ to $\RR$. This allows us to start equation \eqref{eq:SHE} from any time $t_0\in \RR$ by simply making the appropriate integration domain changes in time and the solution adapted to the filtration $\mc{F}_{T}:=\sigma(W((-\infty,T]\times M))\vee \mathcal{N}$.\\
    A random field \(\set{Z(t,x)}_{t\in \RR,x\in M}\) is a \textbf{stationary solution} of \eqref{eq:SHE} if \begin{enumerate}
        \item For any $t<t'\in \RR$, \(Z(t',\cdot)\) is the mild solution to \eqref{eq:SHE} starting from $Z(t,\cdot)$ at time $t'-t$, that is, for all $x\in M$, \[Z(t',x)=P_{t'-t}Z(t,\cdot)(x)+\int_{t}^{t'}\int_M p_{t'-s}(x,y)f(Z(s,y))W(ds,dy).\]
        \item For any $t,t'\in \RR$, we have as random fields in space \(Z(t,\cdot)\stackrel{\mathrm{law}}{=}Z(t',\cdot)\). In other words, for any $n\in \NN_+$, $\{x_i\}_{i=1}^n\subset M$ and $\{y_j\}_{j=1}^n\subset \RR$, we have for all $t>0$ \[\PP[Z(t,x_1)\le y_1,\dots, Z(t,x_n)\le y_n]=\PP[Z(t'
        ,x_1)\le y_1,\dots, Z(t',x_n)\le y_n].\]
    \end{enumerate}
    The law of a stationary solution $Z$ at any time $t$ is an \textbf{invariant measure} of \eqref{eq:SHE}.
\end{definition}

\begin{remark}
    For connections to Da Prato's infinite dimensional evolution point of view, we refer the reader to \cite{COTX25}.
\end{remark}

\subsection{Literature Review}

Equation \eqref{eq:SHE} is interesting on its own\cite{Molchanov1991IdeasIT} and for it's connection to various other areas of probability theory and mathematical physics. In the linear case $f(u)=u$, it is the partition function of directed polymers in a random environment\cite{ACQ11,Comets2004DirectedPI,Cosco2020DirectedPO,DasZhu,DGLpolymer} and gives rise to the Cole--Hopf solution of the KPZ equation \cite{ Ka87, KPZ86}. It also has direct connections to the stochastic Burgers equation \cite{BCJ94,BertiniGiacomin97,CM94} and to Majda’s model of shear-layer flow in turbulent diffusion \cite{Mj93}. In the last few decades, there has been significant progress in understanding fundamental properties of \eqref{eq:SHE} on flat spaces \(\RR^d\) and \(\TT^d\), including intermittency (see, e.g., \cite{ChenKim19, Kh14} and references therein), as well as energy landscapes and fluctuations \cite{DGK23,GHL25,GK24,GL20}. In particular, invariant measures have been heavily studied\cite{COTX25,DGRZ18,DM24,DE24,GQ24,GRSS25} due to their connections with fluctuation theory and other deeper properties. Notably, it has been shown that nontrivial invariant fields exist in the weak disorder/diffusive regime and this has been connected to some of the aforementioned fluctuation results.

In contrast, the literature on non-Euclidean settings has only begun to develop recently, almost all of it for the linear case and focused on well-posedness and/or pointwise moment or almost sure asymptotics. On compact Riemannian manifolds, \cite{TV02,CO25} showed that strong disorder/localization always occurs, the latter also dealing with difficulties introduced by geometry for well-posedness if one allows measure-valued initial conditions. Well-posedness theory for bounded initial conditions, on the other hand, have been extended to stranger geometries such as sub-Riemannian manifolds\cite{BOTW-Heisenberg} and fractals \cite{BCHOTW,HY16}. All the above works have also dealt with moment estimates, which are studied for their connection to intermittency/localization. Most recently, the effect of negative curvature on moment and pointwise asymptotics have begun to be explored. In \cite{GX25,GWX25}, it was shown that for the linear case with noise independent of time, moment asymptotics in Hyperbolic space do not behave differently compared to the Euclidean case, but almost sure pointwise asymptotics grow at a different rate. For the white in time noises (the setting of this paper), \cite{BCO24} showed that for a natural class of noises, negative curvature generates a weak disorder regime. This result was sharpened in \cite{geng-ouyang}, who showed that all noises with $R$ polynomially decaying at rate $r>1$ gives rise to such a regime. Remarkably, they discovered that if $0<r<1$ instead, in place of a weak disorder regime a new regime where strong disorder (\(\lim_{t\to+\infty}\EE[u(t,x)^2]=\infty\)) still holds, but the rate of growth is sub-exponential as opposed to exponential. This phenomenon has never been observed before in any models of the linear equation with white in time noise.

Motivated by these developments, one may ask whether the invariant measures of \eqref{eq:SHE} in the weak disorder (diffusive) regime is connected to the structure of harmonic functions on the underlying space. In Euclidean settings, this question is trivial due to Liouville’s theorem, which forces all bounded harmonic functions to be constant and thus yields a unique spatially homogeneous stationary state. In contrast, on non-compact spaces with nontrivial Martin boundary, such as negatively curved manifolds\cite{Ancona87,AndersonSchoen85} or trees\cite{Croydon2008}, the abundance of bounded positive harmonic functions suggests the possibility of a rich structure of invariant measures reflecting boundary geometry. Theorem~\ref{thm:main} answers this question in a precise and complete form: invariant random fields with uniformly bounded second moments are in one-to-one correspondence with bounded positive harmonic functions. In other words, the invariant measures inherit the structure possessed by the bounded positive harmonic functions, establishing a connection with Martin boundary theory. To our knowledge, this provides the first example of a classification of invariant measures for a stochastic PDE in terms of the harmonic (or boundary) structure of the underlying space.

As mentioned before, the recent work \cite{JKKM26} classified all annealed ergodic measures of \eqref{eq:SHE} with bounded second moment without the assumption $f(0)=0$ on $\RR^d,d\ge3$. They had to deal with many measure theoretic issues, such as proving that invariant measures are supported on genuine random fields and establishing rigidity via spatial averaging. Their classification is parameterized by the mean, which corresponds to constant functions, i.e.\ the bounded harmonic functions on $\mathbb{R}^d$ by Liouville’s theorem. Thus, in the $f(0)=0$ case considered here, the two approaches lead to consistent parameterizations, although by very different mechanisms. It remains open whether the annealed ergodic measures identified in the aforementioned work exhaust all invariant measures even on $\mathbb{R}^d$ without the assumption $f(0)=0$. The present results suggest that, within the $L^2$-bounded class considered here, no additional invariant measures arise. It is also unclear whether such agreement between harmonic and ergodic classifications persists on more general spaces, where both the Liouville property and translation invariance are absent.

\subsection{Theorem \ref{thm:main} on Spaces with Nontrivial Martin Boundary}\label{sec:martin bd}

Invariant measures for stochastic heat equations and related SPDEs have recently attracted significant interest in Euclidean settings \cite{GL20,DM24,GQ24,GHL25,COTX25}. However, Euclidean spaces satisfy the Liouville property, leading to a one-parameter family of invariant measures at fixed $\beta$. If Theorem \ref{thm:main} is nontrivially valid on spaces with nontrivial Martin boundary, we obviously have a much richer classification due to the much richer nature of the family of bounded Harmonic functions. We will discuss this in some detail in this section. For brevity, the proofs of all results in this section will not be included here, but they are present in the arXiv version.

\noindent\textbf{Martin boundary and boundary integral representation}\\
For detailed discussions of Martin boundary theory, see \cite[Ch14]{chung2005markov} and \cite[Ch7]{Pinsky_1995}.\\
Assume that the Markov process associated with $\mathcal{E}$ is transient and admits a Green kernel
\[
G(x,y):=\int_0^\infty p_t(x,y)\,dt <\infty
\quad \text{for $\mu$-a.e.\ $x\neq y$.}
\]
Fix a base point $o\in M$. The \emph{Martin kernel} is defined by
\[
K(x,y):=\frac{G(x,y)}{G(o,y)}.
\]
The \emph{Martin compactification} $\overline M$ is the smallest compactification for which the maps
$y\mapsto K(\cdot,y)$ extend continuously, and the \emph{Martin boundary} is $\partial M:=\overline M\setminus M$. $\partial_M$ is in fact independent of the base point chosen, and is thus well-defined.

Under standard hypotheses (e.g.\ those ensuring solvability of the Dirichlet problem at infinity and the boundary Harnack principle in various concrete models), every $h\in \mathrm{Har}_b^+(M)$ admits a boundary integral representation
\begin{equation}\label{eq:martin-rep}
h(x)=\int_{\partial M} K(x,\xi)\,\nu(d\xi),
\end{equation}
for a finite positive measure $\nu$ on $\partial M$; moreover extremal rays of $\mathrm{Har}_b^+(M)$ correspond to minimal Martin boundary points.

Assumption \eqref{assump:WD} holds on $\RR^d$ with $R\in C_c^\infty(\RR^{2d})$, it holds for $d\ge3$. \cite{BCO24,geng-ouyang} constructs examples of $R$ on negatively curved manifolds. The following gives a general construction using the semigroup kernel.

\begin{proposition}[Heat-kernel mixture covariances satisfy Assumption~\ref{assump:WD}]\label{prop:Lambda_heat_kernel_mixture}
Let $\nu$ be a finite nonnegative Borel measure on $\RR_+$ and define
\[
R_\nu(y,y') := \int_{\RR_+} p_s(y,y')\,\nu(ds).
\]
Then $R_\nu$ is symmetric and nonnegative definite.

Moreover, if
\begin{equation}\label{eq:Hnu_condition}
\int_0^\infty H(s)\,\nu(ds) < \infty,
\qquad
H(s):=\int_s^\infty \sup_{z\in M} p_u(z,z)\,du,
\end{equation}
then Assumption~\ref{assump:WD} holds for $R_\nu$, and in fact if $\nu\neq 0$
\[
0<\Lambda=\Lambda(R_\nu)\le \frac12 \int_{[0,\infty)} H(s)\,\nu(ds).
\]
\end{proposition}
\begin{proof}
Positivity is immediate, since each $p_s(\cdot,\cdot)$ is positive definite and $R_\nu$ is a nonnegative mixture of such kernels.

For the Assumption~\ref{assump:WD} quantity, by Fubini and the semigroup property,
\begin{align*}
&\int_0^\infty \iint_{M^2} p_t(x,y)\,p_t(x',y')\,R_\nu(y,y')\,dydy'\,dt \\
=&
\int_0^\infty \int_0^\infty \iint_{M^2} p_t(x,y)\,p_t(x',y')\,p_s(y,y')\,dydy'\,dt\,\nu(ds) \\
=&
\int_0^\infty \int_0^\infty p_{2t+s}(x,x')\,dt\,\nu(ds).
\end{align*}
Using $p_u(x,x')\le \sup_{z\in M}p_u(z,z)$ and the change of variables $u=2t+s$ gives
\[
\int_0^\infty p_{2t+s}(x,x')\,dt
\le
\frac12\int_s^\infty \sup_{z\in M} p_u(z,z)\,du
=
\frac12 H(s).
\]
Taking the supremum over $x,x'$ yields
\[
\Lambda(R_\nu)\le \frac12\int_0^\infty H(s)\,\nu(ds),
\]
which is finite by \eqref{eq:Hnu_condition}.

Finally, if $\nu\neq 0$, then $R_\nu\not\equiv 0$, and since the heat kernel is nonnegative, the integrand in Assumption~\ref{assump:WD} is positive on a set of positive measure. Hence $\Lambda(R_\nu)>0$.
\end{proof}

\begin{remark}
\eqref{eq:Hnu_condition} is only possible if the Markov process associated with $\mathcal{E}$ is transient, which is necessary for the Martin boundary of a space to be defined. Thus, Proposition \ref{prop:Lambda_heat_kernel_mixture} says that weak disorder and Martin boundary coexist on a large class of spaces, and Theorem \ref{thm:main} connects them via harmonic functions.
\end{remark}

\begin{remark}
For negatively curved manifolds with $\mathcal{L}$ being the Laplace-Beltrami operator, one can often identify $\partial_M$ with the geometric boundary at infinity. 
For Cartan-Hadamard manifolds with negatively pinched curvature, this is done in \cite{AndersonSchoen85}. Various generalizations were given by \cite{Ancona87,Kifer}, while the most generous current upper and lower radial curvature bounds are given in \cite{Ji} for general dimensions and \cite{Neel} for the special case of surfaces. 
Analogously, for trees (and more generally many transient graphs) $\partial_M$ can often be identified with the space of ends; see \cite{Hong2024,PicardelloWoess1987} for more precise conditions.
\end{remark}

\begin{remark}[Nontriviality of the assumption on $R$]
    Assumption \eqref{assump:WD} has already been shown to be true in various settings, all non-compact. In Euclidean space $\RR^d$ with $\mathcal{L}=\Laplace$ and $R$ bounded and compactly supported, it holds in $d\ge 3$. More interestingly, \cite[Lemma 8]{BCO24} showed that on Cartan-Hadamard manifolds with a negative section curvature upper bound and $\mathcal{L}=\Laplace_M$ the Laplace-Bemtrami operator, it holds for $R$ with Riesz-type covariance. For the special case $M=\HH^d$, \cite[Lemma 2.4]{geng-ouyang} identified the sufficient condition $|R(x,y)|\leq Cd(x,y)^{-\alpha},$ where $\alpha>1$ for the assumption to hold. These last two examples have very interesting Martin boundaries at infinity, and the verification comes down to exponential in time heat kernel bounds on diagonal. This means that extensions to non-manifold settings is also simply a matter of constructing appropriate $R$ and looking at heat kernel behavior. Notably, on regular metric trees, it should be easy to verify that the assumption holds for space-time white noise (\(R(x,y)=\delta_{x=y}\)) using known heat kernel upper bounds\cite{Croydon2008,AthreyaLoehrWinter2011,Hong2024}. This is expected to generalize to all non-amenable metric graphs. While these are all examples where $\mathcal{E}$ is local, the theorem does not require that, and indeed we expect that there are examples of nonlocal spaces and appropriate $R$ satisfying the fundamental assumption.
\end{remark}

\begin{conjecture}[Martin Representation for Invariant Measures of Linear Equation]
    In the linear case $f(u)=u$, the solution of \eqref{eq:SHE} starting from $u_0$ admits the formal Feynman-Kac representation \[u(t,x)=\EE_x\left[u_0(B^x_t)\exp\left(\beta\int_0^tW(t-s,B^x_s)ds-\frac{\beta^2t}{2}R(x,x)\right)\right]\] where $B^x$ is the Markov process associated with $\mathcal{E}$ starting from $x$ independent of $W$ and $\EE_x$ denotes expectation taken with respect to $B^x$.
    Observe using the above representation that if the linear \eqref{eq:SHE} starting from $K(\cdot,\xi)$ at fixed $\xi\in \partial_M$ is well-posed, the only dependence on $\xi$ is in the initial condition factor, so by integrating in $\xi$ and taking the $t\uparrow\infty$ limit one should be able to obtain a Martin representation for invariant measures of linear \eqref{eq:SHE}, provided invariant measures with mean $K(\cdot,\xi)$ for all fixed $\xi$ exist. Note however that the latter falls outside of the scope of Theorem \ref{thm:main}, since all $K(x,\xi)$ are unbounded in $x$. At this time, it is unclear what to expect for nonlinear $f$.
\end{conjecture}

\noindent\textbf{Equivariance under symmetry groups}\\
Many well-known examples of spaces with nontrivial Martin boundary also possess large symmetry groups. Since such symmetries often play a fundamental role in the associated potential theory, it is natural to ask whether they are inherited by the family of stationary solutions constructed in Theorem \ref{thm:main}. The following discussion records affirmative answers to that question.

Let $\mathcal G$ be a group acting measurably on $M$. We assume the action preserves the reference measure $\mu$ and intertwines the semigroup:
\[
p_t(gx,gy)=p_t(x,y)
\quad\text{for all }g\in\mathcal G,\ t>0,
\]
and that the spatial covariance kernel $R$ of the driving noise is $\mathcal G$-invariant:
\[
R(gx,gy)=R(x,y)
\quad\text{for all }g\in\mathcal G.
\]

\begin{proposition}[\(\mc{G}-\)equivariance of solutions to \eqref{eq:SHE}]
\label{prop:isom-equiv-mild}
Define the $\mathcal{G}$-action on functions (and random fields) by
\[
(g\cdot \varphi)(x):=\varphi(g^{-1}x),\qquad g\in \mc{G}.
\]
Assume the spatial covariance kernel $R$ of the noise is $G$-invariant.
Define the transformed noise by
\[
W^g(\phi):=W(\phi^g),\qquad \phi^g(t,x):=\phi(t,gx),
\]
so that formally $\dot W^g(t,x)=\dot W(t,g^{-1}x)$.

Let $u$ solve \eqref{eq:SHE} starting from $u_0\in L^\infty(M)$.
Then for each $g\in \mathcal{G}$ the shifted field
\[
u^g(t,x):=(g\cdot u(t,\cdot))(x)=u(t,g^{-1}x)
\]
is the solution of \eqref{eq:SHE} starting from $g\cdot u_0$ driven by $W^g$. In particular, since $W^g\stackrel{\mathrm{law}}{=}W$ by $\mc{G}$-invariance of $R$, the law of $u^g$ coincides with the law of the solution driven by $W$ started from $g\cdot u_0$.
\end{proposition}

\begin{proof}
We first show that the assumption on $R$ gives $\mc{G}-$invariance of noise. Indeed, for test functions $\phi,\psi\in C_c^\infty(\RR_+\times M)$,

\begin{align*}
    &\EE[W^g(\phi)W^g(\psi)]=\EE[W(\phi^g)W(\psi^g)]\\
    =&\int_0^\infty\iint_{M^2} \phi(t,gx)\psi(t,gy)\,R(x,y)\,dx\,dy\,dt\\
    =&\int_0^\infty\iint_{M^2}  \phi(t,x)\psi(t,y)\,R(x,y)\,dx\,dy\,dt,
\end{align*}
where we used the change of variables $x\mapsto g^{-1}x$, $y\mapsto g^{-1}y$ and $\mc{G}-$invariance of $R$ and $dxdy$. Hence $W^g\stackrel{\mathrm{law}}{=}W$.

Fix $g\in \mathcal{G}$ and set $u^g(t,x):=u(t,g^{-1}x)$.
Using \(\mc{G}\)-invariance of the heat kernel and $\mu$,
\[
(P_t u_0)(g^{-1}x) = \int_{M} p_t(g^{-1}x,y)\,u_0(y)\,dy
= \int_{M} p_t(x,gy')\,u_0(gy')\,dy' 
= (P_t(g\cdot u_0))(x),
\]
where we changed variables $y=gy'$.

For the stochastic integral term, define the integrand
\[
\Phi_{t,x}(s,y):=\mathbf 1_{[0,t]}(s)\,p_{t-s}(g^{-1}x,y)\,f(u(s,y)),
\]
so that
\[
\int_0^t\int_{M} p_{t-s}(g^{-1}x,y)\,u(s,y)\,W(ds,dy)= W(\Phi_{t,x}).
\]
Now define
\[
\widetilde \Phi_{t,x}(s,y'):=\mathbf 1_{[0,t]}(s)\,p_{t-s}(x,y')\,f(u^g(s,y'))
=\mathbf 1_{[0,t]}(s)\,p_{t-s}(x,y')\,f(u(s,g^{-1}y')).
\]
Using $p_{t-s}(g^{-1}x,y)=p_{t-s}(x,gy)$, we have
\[
\Phi_{t,x}(s,y)=\widetilde \Phi_{t,x}(s,gy),
\]
i.e. $\Phi_{t,x} = (\widetilde \Phi_{t,x})^g$ with the convention
$\psi^g(s,y):=\psi(s,gy)$.
Therefore, by definition of the transformed noise $W^g$,
\[
W(\Phi_{t,x}) = W\big((\widetilde \Phi_{t,x})^g\big) = W^g(\widetilde \Phi_{t,x})
= \int_0^t\int_{M} p_{t-s}(x,y')\,f(u^g(s,y'))\,W^g(ds,dy').
\]
Combining the transformed deterministic and stochastic terms yields
\[
u^g(t,x)
=(P_t(g\cdot u_0))(x)
+\beta\int_0^t\int_{M} p_{t-s}(x,y)\,f(u^g(s,y))\,W^g(ds,dy),
\]
which is exactly \eqref{eq:SHE-mild} with initial condition $g\cdot u_0$ driven by $W^g$.\end{proof} 

\begin{corollary}[Equivariance of invariant fields and invariant measures]
\label{cor:isom-equiv-invariant}
Assume the hypotheses of Proposition~\ref{prop:isom-equiv-mild}. If $Z$ is a stationary solution of \eqref{eq:SHE}, then $g\cdot Z$ is also a stationary solution for every $g\in \mc{G}$.
\end{corollary}

In particular, a \eqref{eq:SHE}-invariant measure with constant mean is fully invariant under the action of $\mathcal{G}$. For general $h \in \mathrm{Har}_b(M)$, the associated \eqref{eq:SHE}-invariant measures transform equivariantly under $\mathcal{G}$ via the natural action \((g \cdot h)(x) = h(g^{-1}x),\) which is also harmonic. A similar result holds for the fluctuation Gaussian fields $G_h$ of Theorem \ref{thm:fluc}.

\begin{remark}
    The proofs of Theorems \ref{thm:main}--\ref{thm:IC determine long time} and the preceding discussion sees no major role played by the metric $d$, which suggests that up to adding some assumptions, analogous results can be obtained on measure spaces (where one discusses Poisson boundary instead of Martin boundary). Since the most abstract formulation of Dalang--Walsh solutions to SPDE so far has been on metric measure spaces with Dirichlet forms\cite{BCHOTW}, we refrain from pursuing this generality for simplicity.
\end{remark}

\subsection{Acknowledgements}
The author thanks Le Chen and Alexander Dunlap for helpful discussions which motivated this work, Fabrice Baudoin for pointing out an error in a previous version, Robert Neel for discussions on the Dirichlet problem at infinity on negatively curved manifolds, and the anonymous reviewers for their careful reading, corrections, and suggestions for improvement of exposition.. This work was supported by a research grant (VIL73729) from Villum Fonden.

\section{Proof of Theorem \ref{thm:main}}\label{sec:pf main}

The proof of item (i) follows from repeatedly applying It\^o-Walsh isometry into the mild formulation \eqref{eq:SHE-mild} along with $\EE[f(u(s,y))f(u(s,y'))]\leq L_f^2\sup_{y\in M}\EE[u(s,y)^2]$ to obtain a Neumann-type series, and allowing time to go to infinity, yielding \[\sup_{t\ge 0,x\in M}\EE[u(t,x)^2]\leq \norm{u_0}_\infty \sum_{k\ge 0}((\beta L_f)^2\Lambda)^k.\] The RHS geometric series converges iff \((\beta L_f)^2 \Lambda<1\).\\

For item (ii), we first show that the mean of an invariant field must be harmonic. It is simple: taking expectation of \eqref{eq:SHE-mild} gives us $$\EE[u(t,x)]=\int_{M}p_t(x,y)u_0(y)dy=P_tu_0(x).$$
If a field is invariant, then its first moment must also be invariant, and the only functions invariant under the heat semigroup are precisely the harmonic functions. This means that any invariant measure must be associated to some harmonic function $h$, but no two harmonic functions can be associated with the same invariant measure.

We are thus left to prove two steps to finish the proof of item (ii) and therefore the entire theorem: an existence step, for which we adapt the method of \cite{GL20} to solutions of \eqref{eq:SHE} starting from bounded harmonic functions, and a uniqueness step.

\subsection{Existence by Backward Time Shifts}
\noindent This version of the argument first showcased in \cite{GL20} is inspired by \cite[Theorem 2.10]{GHL25}.

\begin{theorem}[Pullback construction from harmonic initial condition]
\label{thm:pullback_harmonic}
Suppose $(\beta L_f)^2\Lambda<1$.
For $h\in \mathrm{Har}_b(M)$ and $K>0$, we define $u^{K,h}(t,x)$ as the solution $K-$backward shifted equation starting from $h$, that is, \[\label{def:shift PAM}u^{K,h}(t,x)=h(x)+\beta\int_{-K}^t\int_{M}p_{t-s}(x,y)f(u^{K,h}(s,y))W(ds,dy).\]

Then:

\begin{enumerate}[label=(\roman*)]
\item For every fixed $t\in\RR$ and $x\in M$, the family $\{u^{K,h}(t,x)\}_{K>-t}$ is Cauchy in $L^2(\Omega)$. Hence there exists a limit
\[
Z^{h}(t,x)
:= \lim_{K\to\infty} u^{K,h}(t,x)
\quad \text{in } L^2(\Omega).
\]

\item The random field $Z^{h}(t,x)$ 
is a stationary solution of \eqref{eq:SHE}.

\item For all $x\in M$,
\[
\EE[Z^{h}(t,x)] = h(x).
\]
\end{enumerate}
\end{theorem}

\begin{proof}
For item (i), we first fix $t\in\RR$ and let $K'>K>-t$, and for them define the difference
\[
\delta^{K,K'}(t,x)
:= u^{K',h}(t,x) - u^{K,h}(t,x).
\]
We also define \[A(K,K'):= \sup_{-K'<-K<t,x\in M}\EE[\delta^{K,K'}(t,x)^2].\] 
By the mild formulation \eqref{eq:SHE-mild}, we have \begin{align*}
    \delta^{K,K'}(t,x)=&\beta I(-K,t)+\beta I(-K',-K),\text{ where }\\
    I(-K,t):=&\int_{-K}^t\int_M p_{t-s}(x,y)[f(u^{K',h}(s,y))-f(u^{K,h}(s,y))] W(ds,dy),\\
    I(-K',-K):=&\int_{-K'}^{-K}\int_M p_{t-s}(x,y)f(u^{K',h}(s,y)) W(ds,dy).
\end{align*}
We will now bound the second moments of these stochastic integrals. By It\^o--Walsh isometry, \begin{align*}
    &\EE[I(-K,t)^2]\\
    =&\int_{-K}^t\iint_{M^2} p_{t-s}(x,y)p_{t-s}(x,y')R(y,y')\\
    &\times \EE[\{f(u^{-K',h}(s,y))-f(u^{-K,h}(s,y))\}\{f(u^{-K'}(s,y'))-f(u^{-K}(s,y'))\}] dydy'ds\\
    \leq& \int_{-K}^t\iint_{M^2} p_{t-s}(x,y)p_{t-s}(x,y')R(y,y')dydy'ds\\
    &\times \sup_{s\in [-K,t]}\sup_{y\in M}\EE[\{f(u^{-K',h}(s,y))-f(u^{-K,h}(s,y))\}^2]\\
    \leq& \Lambda L_f^2 A(K,K').
\end{align*}
In the last step, we used Assumption \ref{assump:WD} and $f$ being Lipschitz with $f(0)=0$.
Similarly, \begin{align*}
    &\EE[I(-K',-K)^2]\\
    =&\int_{-K'}^{-K}\iint_{M^2} p_{t-s}(x,y)p_{t-s}(x,y')R(y,y')\EE[f(u^{K',h}(s,y))f(u^{K',h}(s,y'))] dydy'ds\\
    \leq& L_f^2\sup_{s>-K',y\in M}\EE[u^{-K',h}(s,y)^2]\int_{-K'}^{-K}\iint_{M^2} p_{t-s}(x,y)p_{t-s}(x,y')R(y,y') dydy'ds\\
    =& L_f^2\sup_{s>-K',y\in M}\EE[u^{-K',h}(s,y)^2]\left[\int_{-\infty}^{-K}-\int_{-\infty}^{-K'}\right]\iint_{M^2} p_{t-s}(x,y)p_{t-s}(x,y')R(y,y') dydy'ds\\
    \leq& L_f^2\sup_{s>-K',y\in M}\EE[u^{-K',h}(s,y)^2] (F(K)-F(K')),
\end{align*}
where for $-K<t$, \[F(K):=\sup_{x\in M}\int_{-\infty}^{-K}\iint_{M^2} p_{t-s}(x,y)p_{t-s}(x,y')|R(y,y')| dydy'ds.\] By Assumption \ref{assump:WD}, $F(K)$ is Cauchy in $K$ as $K\uparrow\infty$ for $-K<t$. 
Combining, we have \[\EE[\delta^{K,K'}(t,x)^2]\leq (\beta L_f)^2\Lambda A(K,K') +(\beta L_f)^2\sup_{s>-K',y\in M}\EE[u^{-K,h}(s,y)^2] (F(K)-F(K')).\]
Taking sup over $t,x$ in LHS and noting that \(\sup_{s>-K',y\in M}\EE[u^{-K,h}(s,y)^2]\leq C\) by Assumption \ref{assump:WD} gives \[(1-(\beta L_f)^2\Lambda)A(K,K')\leq (\beta L_f)^2C[F(K)-F(K')]\iff A(K,K')\leq\frac{(\beta L_f)^2C[F(K)-F(K')]}{1-(\beta L_f)^2\Lambda}.\]
Since $F(K)$ is Cauchy in large $K$ by Assumption \ref{assump:WD}, this finishes the proof of item (i). Item (ii) follows from the noise having the same distribution on time intervals of the same size followed by a change of variables in time.\\
For item (iii), taking expectations in the mild formulation gives
\[
\EE[u^{K,h}(t,x)]
=
P_{t+K}h(x)=h(x).
\]
Passing to the $K\uparrow\infty$ limit yields item (iii). This completes the proof.
\end{proof}

\subsection{Uniqueness by Backward Contraction}
In Theorem \ref{thm:pullback_harmonic}, we obtained for every $h\in\mathrm{Har}_b(M)$ a stationary solution of \eqref{eq:SHE} with mean $h$. The construction procedure also proves the desired convergence result in the statement of Theorem \ref{thm:main} by an appropriate change of variables. We will now show this stationary solution is unique for each $h\in \mathrm{Har}_b(M)$, completing the proof of Theorem \ref{thm:main}.
\begin{theorem}[Uniqueness of stationary solutions in the weak-disorder class]
    Assume $(\beta L_f)^2\Lambda<1$ and let $V^1,V^2$ be two stationary solutions of \eqref{eq:SHE} driven by the same noise $W$ such that for all $t\in \RR,x\in M$, we have \(\EE[V^1(t,x)]=\EE[V^2(t,x)]\) and \[\sup_{t\in\RR}\sup_{x\in M}\EE[|V^i(t,x)|^2]<\infty,\quad i\in\set{1,2}.\]
    Then $V^1\equiv V^2$ as random fields.
\end{theorem}

\begin{proof}
    Define $\delta(t,x):=V^1(t,x)-V^2(t,x)$. By definition, $\EE[\delta(t,\cdot)]\equiv0$. Since $V^1$ and $V^2$ are driven by the same noise, $(V^1,V^2)$ is stationary in time as a stochastic process on $\mathcal{M}\times\mathcal{M}$, where $\mathcal{M}$ is the space of functions on $M$.\\
    The theorem is proven if $\delta\equiv0$. By the back-shifted mild formulation of $V^i,i\in \set{1,2}$ at $0$ from $-T$, we have \[\label{eq:delta mild}\delta(0,x)=[P_{T}\delta(-T,\cdot)](x)+\beta\int_{-T}^0\int_M p_{-s}(x,y)[f(V^1(s,y))-f(V^2(s,y))]W(ds,dy).\]
    We will first show \begin{equation}\label{lim:heat smgp term 0}A_T(x):=P_{T}\delta(-T,\cdot),\qquad \lim_{T\uparrow\infty}A_T(x)=0.
    \end{equation}
    Since the Walsh integral over $(-T,0]$ has expectation 0 when conditioned on $\mathcal{F}_{-T}$, we have \[\EE[\delta(0,\cdot)|\mathcal{F}_{-T}]=P_{T}\delta(-T,\cdot)](x).\]
    Sending $T\uparrow\infty$ in the above and applying Kolmogorov's 0--1 law, we have \(\mc{F}_{-\infty}:=\cap_{T\geq 0}\mc{F}_{-T}=\{\emptyset,\Omega\}\). Since $(\mathcal F_{-T})_{T\ge0}$ is a decreasing filtration, the martingale convergence theorem backwards gives us \eqref{lim:heat smgp term 0}.\\
    Now let $A(x,y):=\EE[\delta(0,x)\delta(0,y)]$. If $A\equiv 0$, then \(\EE[(\delta(0,x)-\delta(0,y))^2]=0\) for all $x,y$, hence $\delta\equiv 0$ almost surely.
    The rest of this proof is thus dedicated to showing $A\equiv0$.\\
    By It\^o-Walsh isometry, we have \begin{align*}
        A(x,x')=&\EE[A_T(x)A_T(x')]\\
        &+\beta^2\int_{-T}^0\iint_{M^2} p_{-s}(x,y)p_{-s}(x,y')R(y,y')\\
        &\times\EE[\{f(V^1(y))-f(V^2(y))\}\{f(V^1(s,y'))-f(V^2(s,y'))\}]dydy'ds.\\
        \leq& E[A_T(x)A_T(x')]\\
        &+\beta^2\int_{-T}^0\iint_{M^2} p_{-s}(x,y)p_{-s}(x,y')|R(y,y')|\sup_{z\in M}\EE[|f(V^1(s,z))-f(V^2(s,z))|^2]dydy'ds.
    \end{align*}
    Since $f(0)=0$, we have $|f(x)-f(y)|\leq L_f|x-y|$, and stationarity of $(V^1,V^2)$ gives us \[\sup_{z\in M}\EE[|f(V^1(s,z))-f(V^2(s,z))|^2]=\sup_{z\in M}\EE[|f(V^1(0,z))-f(V^2(0,z))|^2]\leq L_f^2 \norm{A}_\infty,\] where \(\norm{A}_\infty:=\sup_{x,y\in M}|A(x,y)|\). This implies
    \begin{align*}
        A(x,x')\leq& E[A_T(x)A_T(x')]\\
        &+\beta^2\int_{-T}^0\iint_{M^2} p_{-s}(x,y)p_{-s}(x,y')|R(y,y')|L_f^2\norm{A}_{\infty}dydy'ds\\
        \leq& E[A_T(x)A_T(x')]+(\beta L_f)^2\Lambda\norm{A}_\infty.
    \end{align*}
    Taking $\sup_{x,y\in M}$ in LHS and sending $T\uparrow\infty$, we have by \eqref{lim:heat smgp term 0} \[\norm{A}_\infty\leq (\beta L_f)^2\Lambda\norm{A}_\infty,\]
    which forces $\norm{A}_\infty=0$ if $(\beta L_f)^2\Lambda<1$. This concludes the proof.
\end{proof}

\section{Proof of Theorem \ref{thm:fluc}}\label{sec:fluc}

By It\^o-Walsh Isometry, we have \[\label{eq:cov-Gh}\EE[G_h(x)G_h(x')]=\int_0^\infty\iint_{M^2}p_t(x,y)p_t(x',y')R(y,y')f(h(y))f(h(y'))dydy't.\]
Define two quantities \[\Delta_\beta(x):=Z^h_\beta(0,x)-h(x),\qquad M_\beta:=\sup_{x\in M}\EE[\Delta_\beta(x)^2].\]
We will show that 
\begin{enumerate}
    \item[{\rm (i)}] $M_\beta=O(\beta^2)$ as $\beta\downarrow0$,
    \item[{\rm (ii)}] \(\EE\big[\big|\frac{\Delta_\beta(x)}{\beta}-G_h(x)\big|^2\big]\leq CM_\beta\).
\end{enumerate}
Proving both of the above would conclude the proof.\\
Recall the backward mild formulation limit identity \[Z^h_\beta(0,x)=\lim_{T\uparrow\infty}u^{T,h}(0,x)=\lim_{T\uparrow\infty}h(x)+\beta\int_{-T}^0\int_M p_{-s}(x,y)f(u^{T,h}(s,y)) W(ds,dy).\]
Applying It\^o-Walsh isometry in the above identity, we obtain \begin{align*}
    \EE[\Delta_\beta(x)^2]=&\beta^2\int_{-\infty}^0\iint_{M^2} p_{-s}(x,y)p_{-s}(x,y')R(y,y')\EE[f(Z^h_\beta(s,y))f(Z^h_\beta(s,y'))]dydy'dt,\\
    \EE[f(Z^h_\beta(s,y))f(Z^h_\beta(s,y'))]\leq&L_f^2\sup_{x\in M}\EE[Z^h_\beta(x)^2]=L_f^2\sup_{x\in M}\EE[(\Delta_\beta(x)+h(x))^2]=L_f^2(M_\beta+\norm{h}^2_\infty),\text{ thus }\\
    \EE[\Delta_\beta(x)^2]\leq&\beta^2 L_f^2\Lambda M_\beta+\beta^2 L_f^2\Lambda\norm{h}^2_\infty.
\end{align*}
Taking sup in $x$ in LHS gives us \[M_\beta\leq \frac{(\beta L_f)^2\Lambda \norm{h}^2_\infty}{1-(\beta L_f)^2\Lambda}=O(\beta^2),\quad \beta\downarrow0.\]
This proves item (i). For item (ii), we have simply \begin{align*}
    &\EE\left[\left(\frac{\Delta_\beta}{\beta}-G_h(x)\right)^2\right]\\
    =& \int_{0}^\infty\iint_{M^2}p_{t}(x,y)p_t(x,y')R(y,y')\EE[(f(Z^h_\beta(s,y)-f(h(y)))(f(Z^h_\beta(s,y')-f(h(y')))]dydydt\\
    \leq& \Lambda \sup_{z\in M}\EE[(f(Z^h_\beta(s,z)-f(h(z)))^2]\leq L_f^2\Lambda M_\beta.
\end{align*} 

\begin{remark}
Even at first order in $\beta$, the limiting covariance \eqref{eq:cov-Gh} depends on the sector $h$ through the profile $f(h)$; in particular, distinct $h$ typically yield distinct Gaussian limits.
\end{remark}

\section{Proof of Theorem \ref{thm:IC determine long time}}\label{sec:long time}

Let $u^h(t,x)$ be the solution of \eqref{eq:SHE} starting from $h$. From the proof of Theorem \ref{thm:main}, the desired result holds for $u^h$. Thus, by the elementary observation \begin{align*}\EE[|u(t,x)-Z^h_\beta(x)|^2]^{\frac{1}{2}}=&\EE[|u(t,x)-u^h(t,x)+u^h(t,x)-Z^h_\beta(x)|^2]^{\frac{1}{2}}\\
\leq& \EE[|u(t,x)-u^h(t,x)|^2]^{\frac{1}{2}}+\EE[|u^h(t,x)-Z^h_\beta(x)|^2]^{\frac{1}{2}},
\end{align*}
we need only show \(\EE[|u(t,x)-u^h(t,x)|^2]^{\frac{1}{2}}\stackrel{t\uparrow\infty}{\longrightarrow}0.\)

Set $\delta(t,x) := u(t,x) - u^h(t,x)$. By \eqref{eq:SHE-mild},
\[
\delta(t,x)
=
(P_tu_0 - h)(x)
+
\beta \int_0^t \int_M p_{t-s}(x,y)\big(f(u(s,y)) - f(u^h(s,y))\big)\,W(ds,dy).
\]

Applying It\^o-Walsh isometry and using $f$ is Lipschitz, we obtain
\begin{align*}
    \EE[\delta(t,x)^2]\le&|(P_tu_0 - h)(x)|^2\\
    &+(\beta L_f)^2 \int_0^t \iint_{M^2} p_{t-s}(x,y)p_{t-s}(x,y') |R(y,y')| \EE[|\delta(s,y)\delta(s,y')|]\,dydy'\,ds.
\end{align*}

Denote by \(
M(t):=\sup_{x\in M} \EE[\delta(t,x)^2],\quad a(t):=\sup_{x\in M} |(P_tu_0 - h)(x)|^2,\) and \[K(r):=\sup_{x\in M}\iint_{M^2}p_r(x,y)p_r(x,y')|R(y,y')|dydy'.\]

Thus taking sup in $x$ and applying Cauchy-Schwarz in $L^2(\Omega)$, we have \begin{equation}\label{est:M(t)}
    M(t)\leq a(t)+(\beta L_f)^2\int_0^t K(t-s)M(s)ds=a(t)+(\beta L_f)^2\int_0^t K(s)M(t-s)ds.
\end{equation}

Since $a(t)\to 0$ by assumption and $(\beta L_f)^2 \Lambda < 1$ and $\int_0^tK(s)ds\uparrow \Lambda$ as $t\uparrow\infty$ by Assumption \ref{assump:WD}, we see after taking $\sup_{0\le s\le t}$ on both sides of the previous inequality that
\[
\sup_{0\le s\le t} M(s) \le \frac{\sup_{t\ge 0}a(t)}{1-(\beta L_f)^2 \Lambda}<\infty ,
\] where the finitude of the last numerator follows from the fact that $P_t$ is a Markov semigroup.
Thus, \(\overline{M}:=\limsup_{t\uparrow\infty}M(t)\) exists, and we will show \(\overline{M}=0\). Writing \[\int_0^t K(s)M(t-s)ds=\int_0^\infty \mathbb{1}_{s\leq t}K(s)M(t-s) ds=:\int_0^\infty F_t(s)ds,\]
we have by $M\in L^\infty(\RR_+)$ and Assumption \ref{assump:WD} that \(0\leq F_t(s)\leq K(s)\norm{M}_{\infty}\). In fact, for all $s\in \RR_+$, \[\limsup_{t\uparrow\infty}M(t-s)=\overline{M}\implies \limsup_{t\uparrow\infty}F_t(s)=K(s)\overline{M}.\]
Thus by the reverse Fatou lemma, we have \[\limsup_{t\uparrow\infty}\int_0^\infty F_t(s)ds=\int_0^\infty \limsup_{t\uparrow\infty}F_t(s)ds=\overline{M}\int_0^\infty K(s)ds=\Lambda \overline{M}.\]
We now take $\limsup$ in time on both sides of \eqref{est:M(t)}, which gives us \[\overline{M}\leq (\beta L_f)^2\Lambda\overline{M}\implies\overline{M}=0,\] where the implication of the desired conclusion uses the assumption $(\beta L_f)^2\Lambda<1$.

\begin{remark}
The assumption that $P_t u_0$ converges uniformly to a harmonic function $h$ is natural in geometrically non-compact settings. In particular, on spaces with a well-understood Martin boundary (e.g.\ hyperbolic spaces and more generally negatively curved manifolds), bounded functions which admit a continuous extension to the boundary have harmonic limits given by the corresponding Poisson/Martin representation. In such cases, the above proposition shows that the long-time behavior of the stochastic equation is completely determined by the boundary data of the initial condition.
\end{remark}

\printbibliography
\end{document}